\documentclass[10pt,twoside,leqno]{amsart}

\newcommand{\ccH}{\mathcal{X}}

\newcommand{\cnn}{\delta_P}
\newcommand{\mb}{\mathbb}
\newcommand{\C}{\mathbb{C}}
\newcommand{\R}{\mathbb{R}}

\newcommand{\ra}{\to}

\newcommand{\p}{\partial}

\newcommand{\ind}{\hskip .1em\mathrm{ind}}

\newcommand{\var}{\varepsilon}

\newcommand{\ran}{\mathrm{Range}}
\newcommand{\re}{\hskip .1em\mathrm{Re}}

\newcommand{\sign}{\mathrm{sign}\hskip 1pt}
\newcommand{\J}{\mathrm{i}\R}

\newcommand{\no}[1]{|\hskip -1pt | #1 |\hskip -1pt|}

\newcommand{\coker}{\mathrm{coker}}
\newcommand{\codim}{\hskip .1em\mathrm{codim}\hskip .1em}
\newcommand{\Sf}{\mathrm{sf}}
\newcommand{\set}[2]{\left\{\hskip .1em {#1}:{#2}\hskip .1em\right\}}

\newcommand{\Fred}{\mathrm{Fred}}

\def\polk#1{\setbox0=\hbox{#1}{\ooalign{\hidewidth
  \lower1.5ex\hbox{`}\hidewidth\crcr\unhbox0}}}
\usepackage{graphicx}
\usepackage{pstcol}
\usepackage{layout}
\usepackage{multicol}
\usepackage[all]{xy}
\usepackage[english]{babel}
\usepackage{amssymb}
\usepackage{amsmath}
\usepackage{amscd}
\usepackage{mathrsfs}
\usepackage{amsthm}
\usepackage{ifthen}
\usepackage{comment}
\usepackage{xcomment}
\usepackage{setspace}
\usepackage{yfonts}

\newcommand{\im}{\hskip .1em\mathrm{Im}\hskip .1em}

\newcommand{\prc}{p}

\renewcommand{\Im}{\mathrm{im}}
\newcounter{counterxapp}
\newtheorem{remark}{Remark}[section]
\newtheorem{lemma}[remark]{Lemma}
\newtheorem{theorem}[remark]{Theorem}
\newtheorem{prop}[remark]{Proposition}

\newtheorem{cor}[remark]{Corollary}
\theoremstyle{definition}
\newtheorem{defn}[remark]{Definition}
\newtheorem{example}[remark]{Example}

\newcounter{Example}
{\vskip .2em}
\DeclareMathAlphabet{\mathpzc}{OT1}{pzc}{m}{it}
\usepackage{amsmath}
\usepackage{amssymb}
\usepackage[colorlinks]{hyperref}
\usepackage{mathrsfs}
\usepackage{pstricks}
\usepackage{url}
\usepackage[all]{xy}
\excludecomment{TAKEOUT}
\includecomment{TAKEIN}
\begin{document}
\dedicatory{To my parents and my sister}
\address[Garrisi Daniele]{%
\vskip 1pt
Republic of Korea,\vskip 1pt
Gyeongbuk 790784\vskip 1pt
Pohang-Si, Namgu, Hyojadong, San 31\vskip 1pt
Postech, Pohang Mathematics Institute}
\tableofcontents
\title[Spectral flow in Banach spaces]%
{On the spectral flow for paths of essentially hyperbolic bounded
operators on Banach spaces}

\begin{abstract}
We give a definition of the spectral flow for paths of 
bounded essentially hyperbolic operators on a Banach space. The spectral flow 
induces a group homomorphism on the fundamental group of every connected 
component of the space of essentially hyperbolic operators. We prove that this 
homomorphism completes the exact homotopy sequence of a Serre fibration.
This allows us to characterise its kernel and image and to produce examples of 
spaces where it is not injective or not surjective, unlike
what happens for Hilbert spaces. For a large class of paths, namely
the essentially splitting, the spectral flow of $ A $ coincides with 
$ -\ind(F_A) $, the Fredholm index of the differential operator
$ F_A (u) = u' - A u $.
\end{abstract}
\author[Garrisi,~D.]{Garrisi Daniele}
\thanks{This work was supported by the Priority 
Research Centers Program through the National Research Foundation of Korea 
(NRF) funded by the Ministry of Education, Science and Technology 
(Grant 2009-0094068) and the Scuola Normale Superiore of Pisa}
\keywords{Spectral flow, projectors, hyperplanes}
\subjclass{46B20,58B05,58E05,58J30}
\maketitle
\section*{Introduction}
The spectral flow first appeared in \cite{APS75} for a family
of elliptic and self-adjoint operators $ A_t $, ascribed to the
joint work of M.~Atiyah and G.~Lusztig. We outline their effective description
as ``net number of eigenvalues that change sign (from $ - $ to $ + $)
while the parameter family is completing a period'' in the definition
given by J.~Robbin and D.~Salamon in \cite[Theorem~4.21]{RS95}: In
a neighbourhood $ [t_-,t_+] $ of the real line of a point $ t\in\R $ 
(called a crossing) such that $ 0\in\sigma(A_t) $ (the spectrum of $ A_t $),
$ \sigma(A_s) $ can be described as a finite family of continuously 
differentiable curves 
$ \lambda_i\colon [t_-,t_+]\ra\R $ such that $ \lambda_i ' (t)\neq 0 $.
The contribution to the spectral flow of a crossing is given by 
\[
\sum_{\lambda_i} \sign(\lambda_i (t_+)) - \sign(\lambda_i (t_-))
\]
and the spectral flow is the sum of these contributions over all the
crossings. The spectral flow was used to define a Morse index for
1-periodic hamiltonian orbits, as in \cite[pp.~23]{Sal99}, and then to define 
the Floer homology.
The definition by J.~Robbin and D.~Salamon in \cite{RS95} 
requires some differentiability hypotheses and transversality conditions.
P.~Rabier extended their work to unbounded families of operators in Banach 
spaces in \cite{Rab04}.
In \cite{Phi96}, J.~Phillips simplified their definition as follows: 
$ A_t \in\mathcal{F}\sp{sa} (H) $ 
is assumed to be a continuous path of Fredholm, bounded and self-adjoint 
operators, on $ [0,1] $. If $ U $ is a neighbourhood of the origin and 
$ J = [t_-,t_+] $ is a closed interval such that 
\begin{itemize}
\item[1)] $ \sigma(A_s)\cap \p U = \emptyset $ for every $ s\in J $;
\item[2)] $ \sigma(A_s)\cap U $ is a finite set of eigenvalues,
\end{itemize}
then the contribution to the spectral flow from the interval $ J $ is defined 
as 
\[
\dim(P(A(t_+);U)) - \dim(P(A(t_-);U))
\]
where $ P(A;U) $ is the spectral projector of $ A $ relative to $ U $.
The spectral flow is the sum of all the contributions
obtained over a partition of the unit interval,
$ J_i $, such that a neighbourhood $ U_i $ as in 1,2) corresponds to each 
$ J_i $. It is
invariant for fixed-endpoint homotopies and defines a groups homomorphism on 
the fundamental group of each connected component of 
$ \mathcal{F}\sp {sa} (H) $. If $ H $ is infinite-dimensional and
separable, then there are exactly three connected components, corresponding to 
$ I $ and $ -I $ (these are contractible to a point) and to $ 2P - I $, where 
$ P $ is a projector with infinite-dimensional kernel and image. 
On the fundamental group of the third one, denoted 
by $ \mathcal{F}\sp{sa} _* $, the spectral flow
\[
\Sf\colon\pi_1 (\mathcal{F}\sp{sa} _*)\ra\mb{Z}
\]
is a group isomorphism. In \cite{BBLP05},~J.~Phillips later extended 
this definition to continuous families of unbounded, 
Fredholm and self-adjoint operators. In contrast with \cite{RS95}, no
assumptions of differentiability or transversality are made.
C.~Zhu and Y.~Long in \cite{ZL99} extended the definition in \cite{Phi96}
to bounded, admissible operators (which are compact perturbations of hyperbolic
operators) on Banach spaces: On every interval 
$ J_i = [t_- \sp i,t_+ \sp i ] $ of a suitable partition of $ [0,1] $, they 
provide a path of projectors $ Q_i $ such that
\begin{gather*}
Q_i\colon J_i\ra\mathcal{P}(\mathcal{L}(E))\text{ is continuous},\\
Q_i (t) - P\sp + (A(t))\text{ is compact for } t\in J_i,\\
Q_i (t)\text{ is a spectral projector of } A(t).
\end{gather*}
The contribuition of $ J_i $ to the spectral flow is
\[
[Q_i (t\sp i _-) - P\sp + (A(t\sp i _-))] - 
[Q_i (t\sp i _+) - P\sp + (A(t\sp i _+))],
\]
where $ [Q - P] $ denotes the Fredholm index of 
$ P\colon\ran(Q)\ra\ran(P) $. The spectral flow is defined as the sum
of the quantity above as $ J_i $ varies over the partition of $ [0,1] $.
We denote the spectral projector relative to the positive 
complex half-plane by $ P\sp + (A) $, and 
$ \mathcal{P}(\mathcal{L}(E)) $ is the set of 
projectors of $ E $. This work has three essential purposes:\par
I) Further simplifying the definition of spectral flow. 
Given a path $ A_t \in e\mathcal{H}(E) $ on $ [0,1] $ of essentially 
hyperbolic operators 
(which are compact perturbations of hyperbolic operators and correspond to
the admissible ones used in \cite{ZL99}), there exists a 
continuous path $ P $ of projectors on $ [0,1] $
such that 
\[
P(t) - P\sp + (A(t))\text{ is compact }
\]
for every $ t\in [0,1] $. Thus, we define
\[
\Sf(A) = [P(0) - P\sp + (A(0))] - [P(1) - P\sp + (A(1))].
\]
Therefore, we do not need to partition the unit interval (as in
the definition in \cite{ZL99}), as long as
we do not require $ P(t) $ to be a spectral projector of $ A(t) $.\par
The existence of such a path $ P $ follows from the 
homotopy lifting property of the Serre fibration
\[
\prc\colon \mathcal{P}(\mathcal{L}(E))\ra\mathcal{P}(\mathcal{C}(E)).
\]
We use $ \mathcal{C}(E) $ to denote the quotient of the operator
algebra $ \mathcal{L}(E) $ by the ideal of the compact operators.
$ \mathcal{P}(\mathcal{C}(E)) $ is the space of projectors of the Calkin
algebra and $ p $ is the quotient projection.
The definition we give of spectral flow coincides with the one in 
\cite{ZL99}. In \S\ref{Three}, we show that the spectral flow is invariant for 
fixed-endpoint homotopies and that, given two continuous paths
$ A $ and $ B $ such that $ A(1) = B(0) $, there holds
$ \Sf(A*B) = \Sf(A) + \Sf(B) $. These and other basic facts are outlined
in Proposition~\ref{prop:sf-properties}.\par
II) Studying the homomorphism properties of the spectral flow.
We prove that $ e\mathcal{H}(E) $ is homotopically equivalent to 
$ \mathcal{P}(\mathcal{C}(E)) $ and that the spectral flow completes the 
exact homotopy sequence
of the fibration above. This allows us to characterise the kernel and the 
image of the spectral flow $ \Sf_P $ defined on the fundamental group of the 
connected component of $ 2P - I\in e\mathcal{H}(E) $. Precisely: an integer
$ m $ belongs to the image of $ \Sf_P $ if and only if
\begin{itemize}
\item[h1)] there exists a projector $ Q $ connected to $ P $ by a path in the 
space of projectors $ \mathcal{P}(\mathcal{L}(E)) $, such that $ Q - P $ is 
compact and $ [P - Q] = m $;
\end{itemize}
$ \ker(\Sf_P)\cong\Im(\prc_*) $. Hence, $ \Sf_P $ is injective if and only if
\begin{itemize}
\item[h2)] $ \text{im}(\prc_*) = \{0\} $.
\end{itemize}
The existence of a projector satisfying h1,2) depends heavily on the 
structure of the Banach space $ E $.
When $ E $ is Hilbert, J.~Phillips proved in \cite{Phi96} that for every
projector with infinite-dimensional range and kernel, h1), with
$ m = 1 $, and h2) hold.
We show that $ \ell\sp p $ and $ \ell\sp\infty $ have this property as well. 
In general, at least one projector (and then infinitely many)
exists in the spaces satisfying the hypotheses of 
Proposition~\ref{prop:surjective} and Proposition~\ref{prop:injective}.\par
The question whether property h1) holds for some projector
is strongly related to the existence of complemented subspaces isomorphic to 
closed subspaces of co-dimension $ m $. This relation is highlighted
in Proposition~\ref{prop:surjective}. In fact, given a space $ E $, 
isomorphic to its hyperplanes, the projector over each of
the summands of $ E\oplus E $ fulfills property h1). Thus, $ \Sf_P $
is surjective. If $ E $ is isomorphic to its subspaces of co-dimension two, 
but not to its hyperplanes, the image 
of the spectral flow is $ 2\mb{Z} $, and so on.
Examples of such spaces have been constructed by W.~T.~Gowers and B.~Maurey 
in \cite{GM97}. If $ E $ is not isomorphic to any of its proper
subspaces, then $ \Sf_P $ is zero. Such a space was constructed by
W.~T.~Gowers and B.~Maurey in the celebrated paper \cite{GM93}.\par
We prove that in a Douady space (cf. \cite{Dou65}), there are 
projectors $ P $
such that $ \Sf_P $ is not injective and projectors with infinite-dimensional 
range and kernel such that the spectral flow is zero.\par
III) Comparing the spectral flow with the Fredholm index of the 
differential operator
\[
F_A\colon W\sp{1,p} (\R)\ra L\sp p (\R),\ \ u\mapsto
\left(\frac{d}{dt} - A(t)\right)u.
\]
In \S\ref{Five}, we extend the definition of spectral flow for a path
$ A_t\in e\mathcal{H}(E) $ on $ \mathbb{R} $ with hyperbolic limits at
$ \pm\infty $. We prove in Theorem~\ref{thm:final} that for a large class of 
paths, which are essentially hyperbolic, with hyperbolic limits, 
and \textsl{essentially splitting} (cf. \cite{AM03}), the equality 
\[
\ind(F_A) = - \Sf(A)
\]
holds. The equality above applies, for instance, in the special case where 
$ A $ is a continuous, compact perturbation of a path of hyperbolic operators 
with some boundary conditions (check \cite[Theorem~E]{AM03}).
Our theorem confirms the guess of A.~Abbondandolo and P.~Majer in \S7
of \cite{AM03} that for these paths, the equality above holds.
\vskip .4em
We remark that our work deals with \textsl{bounded} operators. The 
differential operator $ F_A $ arises naturally from the linearisation of 
a vector field $ \xi\in C\sp 1 (E,E) $ on a solution 
$ v'(t) = \xi(v(t)) $, such that the endpoints are zeroes of
$ \xi $ and $ A(t) = D\xi(v(t)) $. If $ F_A $ is Fredholm and surjective
(and the zeroes are hyperbolic), then the set
\[
W_\xi (p,q) = \{v\colon\R\ra E:v'(t) = \xi(v(t)),
v(-\infty) = p,v(+\infty) = q\}
\]
is a sub-manifold of dimension $ \ind(F_A) $. This constitutes a 
landmark for the study of the Morse theory on Banach 
manifolds, as in \cite{AM06}. A proof of this can be found in 
\cite[\S8]{AM03}. If $ A $ fulfils the hypotheses of Theorem~\ref{thm:final}, 
then $ \Sf(A) $ determines the dimension of the manifold.\vskip .4em
The spectral flow also provides an index for 1-periodic solutions
of $ u' (t) = X(u(t)) $. If $ DX(u_t) $ is essentially hyperbolic,
then $ \mathrm{index}_X (u) = \Sf(DX (u_t)) $. Thus, in order
to have a good Morse theory for periodic solutions,
the question whether there are loops with 
non-trivial spectral flow becomes relevant.
\vskip .4em
\textsc{Acknowledgements}.
I would like to thank Prof.~Alberto~Abbondandolo
and Prof.~Pietro~Majer for their aid and suggestions, the referees for
their careful reading and for pointing out to me the papers of Yiming~Long
and Zhu~Chaofeng. I also thank my parents and my sister for always 
supporting me.
\section{Preliminaries}
Here we review some basic definitions and results on spectral theory and
Fredholm operators. Our main references are \cite[Ch.~X]{Rud91} and 
\cite[IV.4,5]{Kat95}.
\subsection{Spectral theory}
A \textsl{Banach algebra} is an algebra 
$ \mathcal{A} $ with unit, 1, over the real or complex field and a norm 
$ \no{\cdot} $ such that
\begin{gather*}
\big(\mathcal{A},\no{\cdot}\big)\text{ is a Banach space }\\
\no{xy}\leq\no{x}\cdot\no{y}\text{ for every } x,y\in\mathcal{A}.
\end{gather*}
we denote by $ G(\mathcal{A}) $ the set of invertible elements of the 
algebra; it is an open subset of $ \mathcal{A} $. 
Given $ x\in\mathcal{A} $, the subset of the field
\[
\sigma(x) = \{\lambda\in\mb{F}:x - \lambda\cdot 1\not\in G(\mathcal{A})\}
\]
is called the \textsl{spectrum} of $ x $. Let us review some properties
of the spectrum.
\begin{prop}
\label{prop:spectrum-basics}
For every $ x\in\mathcal{A} $, $ t\in\mb{F} $ and $ \Omega\subset\mb{F} $ 
an open subset,
\begin{enumerate}
\item if $ \sigma(x) $ is non-empty, then it is closed and bounded;
\item $ \sigma(x + t) = \sigma(x) + t $, $ \sigma(tx) = t\sigma(x) $;
\item then there exists $ \delta > 0 $ such that $ \sigma(y)\subset\Omega $
for every $ y\in B(x,\delta) $;
\item if $ \mathcal{A} $ is complex, then $ \sigma(x) $ is non-empty;
\item if $ f\colon\mathcal{A}\ra\mathcal{B} $ is an algebras homomorphism
such that $ f(1) = 1 $, then $ \sigma(f(x))\subseteq\sigma(x) $.
\end{enumerate}
\end{prop}
Given a real algebra, we can consider the complex algebra associated with
it, $ \mathcal{A}_{\C} = \mathcal{A}\otimes_{\R} \C $. We have an inclusion
of algebras 
\[
\mathcal{A}\hookrightarrow \mathcal{A}_{\C},\ \ x\mapsto x\otimes 1
\]
and $ \sigma(x)\subseteq\sigma (x\otimes 1) $. Hereafter, we will
take $ \sigma(x\otimes 1) $ as the spectrum of $ x $. 
\begin{defn}
\label{defn:simple}
A finite family of closed curves in the complex plane, 
$ \Gamma = \{c_i:1\leq i\leq n\} $, is said to be \textsl{simple} if, 
for every $ z\not\in\bigcup_{i = 1}\sp n \Im(c_i) $,
\[
\ind_{\Gamma} (z) :=
\frac{1}{2\pi i}\int_{\Gamma} \frac{d\zeta}{z - \zeta}
= \frac{1}{2\pi i}\sum_{i = 1}\sp n \int_{c_i} \frac{d\zeta}{z - \zeta}
\in\{0,1\}.
\]
we denote by $ \Omega_0 (\Gamma) $ and $ \Omega_1 (\Gamma) $ the subsets of 
the complex space such that $ \ind_{\Gamma} (z) $ is $ 0 $ and $ 1 $, 
respectively.
\end{defn}
\begin{defn}
An element $ p\in\mathcal{A} $ is said to be a \textsl{projector} if 
$ p\sp 2 = p $. We can associate with it the sub-algebra 
$ \mathcal{A}(p) = \{pxp:x\in\mathcal{A}\} $, with the unit $ p $.
we denote by $ \sigma_p (y) $ the spectrum of an element 
$ y\in\mathcal{A}(p) $.
\end{defn}
If $ xp = px $, we have the equality
\begin{equation}
\label{eq:sp}
\sigma(x) = \sigma_p (xp)\cup\sigma_{1 - p} (x(1 - p)).
\end{equation}
\begin{theorem}
\label{thm:spectrum-projector}
Let $ x\in\mathcal{A} $, $ \Sigma\subset\sigma(x) $ open and closed in
$ \sigma(x) $. Then, there exists a projector, called a  
\textsl{spectral projector} relative to $ \Sigma $, which we denote by 
$ P(x;\Sigma) $, such that
\begin{enumerate}
\item $ \sigma_P (PxP) = \Sigma $;
\item $ Px = xP $;
\item $ I - P = P(x;\Sigma\sp c) $;
\item given an algebra homomorphism, $ f\colon\mathcal{A}\ra\mathcal{B} $,
$ f(P(x;\Sigma)) = P(f(x);\Sigma) $.
\end{enumerate}
Given $ \Gamma $ simple such that $ \Sigma = \Omega_1 (\Gamma)\cap\sigma(x) $,
\begin{equation*}
P(x;\Sigma) = P_{\Gamma} (x) 
:= \frac{1}{2\pi i}\int_{\Gamma} (x - \zeta)\sp {-1} d\zeta.
\end{equation*}
On the open subset $ \mathcal{A}(\Gamma) = \{x\in\mathcal{A}:
\sigma(x)\cap\Gamma = \emptyset\} $, $ P_{\Gamma} (x) $ is a continuous map.
\end{theorem}
For the proof and more details check \cite[Theorem~10.27]{Rud91}
or \cite[Theorem~6.17]{Kat95}.
\subsection{Spaces of projectors}
\label{ssect:spaces-of-projectors}
Given Banach spaces $ E $ and $ F $, we denote by $ \mathcal{L}(E,F) $ 
the space of bounded operators from $ E $ to $ F $. When $ E = F $ we
use the notation $ \mathcal{L}(E) $. 
We denote the subsets of compact operators by $ \mathcal{L}_c (E,F) $ and
$ \mathcal{L}_c (E) $.
The composition of operators endows the space $ \mathcal{L}(E) $ 
with the structure of a Banach algebra (the identity operator
being the unit), and the subspace of compact operators is a closed ideal. 
We denote the quotient algebra by
$ \mathcal{C}(E) $. It inherits the structure of a Banach algebra and is 
called the \textsl{Calkin algebra}. The quotient projection
\begin{equation}
\label{eq:quotient}
\prc\colon \mathcal{L}(E)\ra \mathcal{C}(E),\ \ A\mapsto A + \mathcal{L}_c (E)
\end{equation}
is an algebra homomorphism. 
\begin{defn}
Given a Banach algebra $ \mathcal{A} $, we define the following subsets
\begin{enumerate}
\item $ \mathcal{P}(\mathcal{A}) = \{p\in\mathcal{A}:p\sp 2 = p\} $, 
\textsl{projectors};
\item $ \mathcal{Q}(\mathcal{A}) = \{q\in\mathcal{A}:q\sp 2 = 1\} $,
\textsl{square roots of the unit};
\item $ \mathcal{H}(\mathcal{A}) = \{x\in\mathcal{A}:\sigma(x)\cap 
i\R = \emptyset\} $, \textsl{hyperbolic} elements.
\end{enumerate}
\end{defn}
\subsubsection*{Properties and remarks}
$ \mathcal{P} $ and $ \mathcal{Q} $ are closed subsets, locally
path connected and analytic sub-manifolds. A proof of this can be
found in \cite[Lemma~1.5]{AM09}. $ \mathcal{P} $ and $ \mathcal{Q} $ 
are diffeomorphic to each other through the diffeomorphism 
$ p\mapsto 2p - 1 $. By (iii) of 
Proposition~\ref{prop:spectrum-basics}, 
applied with $ \Omega = \{z:\re(z)\neq 0\} $,
the subset $ \mathcal{H}(\mathcal{A})\subset\mathcal{A} $ is open.
We denote by $ G_1 (\mathcal{A}) $ the connected component of 
$ G(\mathcal{A}) $ of the unit. 
\begin{theorem}
\label{thm:pair-of-projectors}
Given two projectors $ p,q $ such that either $ \no{p - q} < 1 $ or both
are in the same connected component of $ \mathcal{P}(\mathcal{A}) $,
there exists $ u\in G_1 (\mathcal{A}) $ such that $ up = qu $.
\end{theorem}
For the proof and details, we refer to 
\cite[Proposition~4.2]{PR87} and \cite[Proposition~2.2]{Gar08}.
The theorem above has two consequences:
\begin{itemize}
\item[(c1)] $ \mathcal{P}(\mathcal{A}) $ is locally path-connected. So
is $ \mathcal{Q}(\mathcal{A}) $;
\item[(c2)] when $ \mathcal{A} = \mathcal{L}(E) $, two projectors in the
same connected component have isomorphic ranges and kernels.
\end{itemize} 
The quotient projection $ \prc $ in (\ref{eq:quotient}) restricts to the 
subset of projectors and roots of the unit
\begin{gather*}
\mathcal{P}(\prc)\colon\mathcal{P}(\mathcal{L}(E))\ra
\mathcal{P}(\mathcal{C}(E)),\ \ 
P\mapsto P + \mathcal{L}_c (E)\\
\mathcal{Q}(\prc)\colon\mathcal{Q}(\mathcal{L}(E))\ra
\mathcal{Q}(\mathcal{C}(E)),\ \ 
Q\mapsto Q + \mathcal{L}_c (E).
\end{gather*}
\begin{defn}
A continuous map $ p\colon E\ra B $ has the
\textsl{homotopy lifting property} w.r.t. a topological space $ X $
if, given continuous maps
\[
h\colon X\times [0,1]\rightarrow B, \ \ \ 
f\colon X\times\{0\}\rightarrow E,
\]
there exists $ H\colon X\times [0,1]\rightarrow E $ such that
$ H(x,0) = f(x,0) $ and $ p\circ H = h $. If the homotopy 
lifting property holds w.r.t. $ [0,1]^n $ for every $ n\geq 0 $, then
$ \prc $ is called a \textsl{Serre fibration}.
\end{defn}
\begin{prop}
\label{prop:loc-sect-Calkin}
The maps $ \mathcal{P}(\prc) $ and $ \mathcal{Q}(\prc) $ are surjective
Serre fibrations.
\end{prop}
In general, every surjective algebra homomorphism induces
a Serre fibration. For a proof, see \cite[Theorem~2.4]{PR90}.
The surjectivity of $ \mathcal{P}(p) $ and $ \mathcal{Q}(p) $ follows from
\cite[Proposition~4.1]{AM09}. In fact, $ \mathcal{P}(\prc) $ and 
$ \mathcal{Q}(\prc) $ are locally trivial fiber bundles, as follows 
from \cite[Proposition~1.3]{AM09} or \cite[Theorem~4.2]{Gar08}.
\subsection{Fredholm operators and relative dimension}
\label{sect:fred}
Let $ T\in\mathcal{L}(E,F) $ be a bounded operator. If the image of
$ T $ is a closed subspace, we have two Banach spaces associated with it,
namely $ \ker(T) $ and $ E/\ran(T) = \coker(T) $. 
\begin{defn}
An operator as above is said to be \textsl{semi-Fredholm} if either 
$ \ker(T) $ 
or $ \coker(T) $ is a finite-dimensional space. If both have finite
dimension, $ T $ is called \textsl{Fredholm} and the integer
\[
\ind(T) = \dim\ker(T) - \dim\coker(T)
\]
is the \textsl{Fredholm index}. Otherwise, the index is defined to be 
$ +\infty $ or $ -\infty $ as long as $ \ker(T) $ or $ \coker(T) $ has 
infinite dimension.
\end{defn}
we denote by $ \Fred(E,F) $ and $ \Fred(E) $ the subsets of Fredholm 
operators in $ \mathcal{L}(E,F) $ and $ \mathcal{L}(E) $, respectively; 
$ \Fred_k (E,F) $ is the set of Fredholm operators of index $ k $. 
\begin{prop}
\label{prop:fredholm-properties}
Let $ T\in\Fred(E,F) $, $ S\in\Fred(F,G) $ and 
$ K\in\mathcal{L}_c (E,F) $. We have
\begin{itemize}
\item[(a)] $ \Fred_k (E,F)\subseteq\mathcal{L}(E,F) $ is an open subset;
\item[(b)] $ T + K\in\Fred(E,F) $ and $ \ind(T + K) = \ind(T) $;
\item[(c)] $ S\circ T \in\Fred(E,G) $ and 
$ \ind(S\circ T) = \ind(S) + \ind(T) $;
\item[(d)] given $ B\in\mathcal{L}(E,F) $, there exists $ \var > 0 $ such
that the maps $ \dim\ker(T + \lambda B) $ and $ \dim\coker(T + \lambda B) $
are constant on $ B(0,\var)\setminus\{0\} $;
\item[(e)] $ T\in\Fred(E,F) $ if and only if there exists 
$ U\in\mathcal{L}(F,E) $ such that $ T $ and $ U $ are the
\textsl{essential inverse} of each other, that is,
$ T\circ U - I\in\mathcal{L}_c (F) $
and $ U\circ T - I\in\mathcal{L}_c (E) $.
\end{itemize}
\end{prop}
Statements (a,e) are easy to check. Statements (a,b,d) are all stated and 
proved in \cite[Ch.~IV.5]{Kat95} in the more general setting of semi-Fredholm 
and unbounded operators.
\begin{defn}
A pair of closed subspaces $ (X,Y) $ is \textsl{semi-Fredholm} if and only if 
their sum is closed and either $ X\cap Y $ or $ E/(X + Y) $  has a 
finite dimension. If both have finite dimension, 
then the \textsl{Fredholm index} of the pair $ (X,Y) $ is defined as
\[
\ind(X,Y) = \dim X\cap Y - 
\codim X + Y.
\]
Otherwise, the index is $ +\infty $ or $ -\infty $, when either $ X\cap Y $ or
$ E/(X + Y) $ has infinite dimension.
\end{defn}
Two projectors $ P,Q $ are compact perturbations of each other
if $ P - Q\in\mathcal{L}_c (E) $. In this case, the restriction of
$ Q $ to $ \ran(P) $ is in $ \Fred(\ran(P),\ran(Q)) $. The 
\textsl{relative dimension} between $ P $ and $ Q $ is defined as
\[
[P - Q] := \ind (Q\colon\ran(P)\ra\ran(Q)).
\]
This definition is meant to generalise the dimension gap between
two finite-dimensional spaces to Banach spaces. The notation above is used by 
C.~Zhu and Y.~Long in \cite{ZL99}. Corresponding definitions are known in
Hilbert spaces, considered by A.~Abbondandolo and P.~Majer in 
\cite[Definition~1.1]{AM01} 
(see also \cite[Remark~4.9]{BDF73}). A definition of relative dimension
for pairs of closed subspaces $ (X,Y) $, not necessarily complemented,
can be found in \cite[Definition~5.8]{Gar08}.
\begin{theorem}
\label{thm:relative-dimension}
Given pairs of projectors $ (P,Q) $ and $ (Q,R) $ with compact difference,
we have
\begin{enumerate}
\item if $ \ran(P) $ and $ \ran(Q) $ have finite dimension, then\\
$ [P - Q] = \dim\ran(P) - \dim\ran(Q) $;
\item $ [P - R] = [P - Q] + [Q - R] $;
\item on the subset $ \{(P,Q)\in \mathcal{P}(\mathcal{L}(E))\times 
\mathcal{P}(\mathcal{L}(E)):P - Q\in\mathcal{L}_c (E)\} $, 
the map $ [P - Q] $ is continuous;
\item $ [P - Q] = [(I - Q) - (I - P)] $;
\item $ (\ran(P),\ker(Q)) $ is a Fredholm pair
and $ \ind(\ran(P),\ker(Q)) = [Q - P] $.
\end{enumerate}
\end{theorem}
Property (iii) follows from stability results for the index of semi-Fredholm 
pairs; see \cite[Remark~IV.4.31]{Kat95} and \cite[Theorem~3.3]{Gar08}. 
For a proof of (iv) and (v), see \cite[Lemma~2.3]{ZL99} and 
\cite[Proposition~5.13]{Gar08}, respectively;
(i) follows from the remarks after Definition~5.8 in \cite{Gar08}.
\section{Essentially hyperbolic operators}
We recall that a bounded operator $ A\in\mathcal{L}(E) $ - 
or more generally an element of a Banach algebra $ \mathcal{A} $ - is called 
\textsl{hyperbolic} if its spectrum does not meet the imaginary axis. 
We denote by $ GL(E) $ the group of invertible operators on $ E $
and by $ GL_I (E) $ the connected component of the identity operator.
\vskip .2em
Given $ A\in\mathcal{L}(E) $, the spectrum of $ A + \mathcal{L}_c (E) $ is 
called the \textsl{essential} spectrum. It is usually denoted by
$ \sigma_e (A) $. By (e) of Proposition~\ref{prop:fredholm-properties},
\begin{equation}
\label{eq:essential-spectrum}
\sigma_e (A) = \set{\lambda}{A - \lambda\not\in\Fred(E)}.
\end{equation}
\begin{defn}
\label{defn:essentially-hyperbolic}
An operator $ A $ is called \textsl{essentially hyperbolic} if 
$ A + \mathcal{L}_c (E) $ is a hyperbolic element in $ \mathcal{C}(E) $.
\end{defn}
By the equality above, an operator $ A \in\mathcal{L}(E) $ is essentially
hyperbolic if and only if its essential spectrum does not meet the imaginary
axis. A consequence of (\ref{eq:essential-spectrum}) is:
\begin{lemma}
\label{lem:structure}
Let $ D(A) $ be the set of all isolated points of 
$ \sigma(A) $, and let $ \partial\sigma(A) $ be the set of 
the boundary points of $ \sigma(A) $. Then 
$ \partial \sigma(A)\setminus D(A) $ is a subset of $ \sigma_e(A) $.
\end{lemma}
\begin{proof}
Let $ \lambda\in\p\sigma(A)\setminus D(A) $, and suppose that 
$ \lambda\not\in\sigma_e (A) $, thus $ A - \lambda\in\Fred(E) $. Let
$ \var > 0 $ as in (d) of Proposition~\ref{prop:fredholm-properties}, 
with $ B = -I $. Therefore, for some $ c,k\in\mb{Z} $
\[
\dim\ker(A - z) = c,\ \dim\coker(A - z) = k,\ \text{ for every }
z\in B(\lambda,\var)\setminus\{\lambda\}.
\]
Because $ \lambda\in\p\sigma(A) $, there exists 
$ w\in B(\lambda,\var)\setminus\{\lambda\} $ such that $ A - w $ is invertible.
Thus, $ c = k = 0 $ and $ A - z $ is invertible for every 
$ z\in B(\lambda,\var)\setminus\{\lambda\} $. Hence $ \lambda $ is isolated 
in $ \sigma(A) $.
\end{proof}
We need a well-known fact about the topology of the real line:
\begin{prop}
A closed proper subset of the real line with an empty boundary is discrete.
\end{prop}
\begin{cor}
If $ A $ is an essentially hyperbolic operator, the
set $ \sigma(A)\cap i\R $ is finite.
\end{cor}
\begin{proof}
We show that the boundary of $ \sigma(A)\cap i\R $ is empty. 
Suppose it is not and let $ \lambda\in\partial(\sigma(A)\cap i\R) $ be an
arbitrary point. Hence, $ \lambda\in\partial\sigma(A) $.
Because $ \sigma(A)\cap i\R $ is closed, $ \lambda\in i\R $. Hence
$ \lambda\not\in\sigma_e (A) $, because $ A $ is essentially hyperbolic. Thus, 
$ \lambda\in\partial\sigma(A)\setminus\sigma_e (A) $, whence, by 
Lemma~\ref{lem:structure}, $ \lambda\in D(A) $. Hence, $ \lambda $ is 
isolated in $ \sigma(A)\cap i\R $
in contradiction with the hypothesis that $ \lambda $ is a boundary point.
By the proposition above, $ \sigma(A)\cap i\R $ is discrete. Because it is 
also compact, it is a finite set.
\end{proof}
\begin{prop}
If $ A $ is an essentially hyperbolic operator, each of the points of 
$ \sigma(A)\cap\J $ is an eigenvalue of finite algebraic multiplicity.
\end{prop}
\begin{proof}
Let $ \lambda\in\sigma(A)\cap\J $. 
We infer that $ A - \lambda\in\Fred_0 (E) $. By (\ref{eq:essential-spectrum}),
$ A - \lambda\in\Fred_k (E) $ for some
$ k\in\mathbb{Z} $. Now, by (a) of Proposition~\ref{prop:fredholm-properties}, 
there exists a neighbourhood $ V $ of $ \lambda $ such that
\[
A - z\in\Fred_k (E),\ \ z\in V.\\
\]
Because $ \lambda $ is isolated, there 
exists $ z'\in V\setminus\{\lambda\} $
such that $ A - z' $ is invertible, hence $ k = 0 $.
Thus, because $ A - \lambda $ is not invertible, 
$ \ker(A - \lambda)\neq\{0\} $.
Hence $ \lambda $ is an eigenvalue and, by hypothesis, isolated.
These two conditions, by Theorem~5.10 and Theorem~5.28 of \cite{Kat95}, imply 
that the spectral projector $ P(A;\{\lambda\}) $ has range of
finite dimension, which is the algebraic multiplicity.
\end{proof}
Theorem \ref{thm:spectrum-projector} provides us with projectors
$ P_i = P(A;\{\lambda_i\}) $ for every $ \lambda_i\in\sigma(A)\cap\J $.
Let $ P = P(A;\sigma(A)\cap\{\re(z)\neq 0\}) $. 
We can write
\begin{equation}
\label{eq:structure}
A = \left(A P + \sum_{i = 1} \sp n P_i\right)  + (A - I)\sum_{i = 1}^n P_i.
\end{equation}
According to (i,ii) of Theorem \ref{thm:spectrum-projector},  
the term in the brackets is hyperbolic. The last term has finite rank.
Thus, we have proved that an essentially hyperbolic operator is a compact
perturbation of a hyperbolic one. Conversely, a compact perturbation of a 
hyperbolic operator is essentially hyperbolic. In fact, let
$ H,K $ be a hyperbolic and a compact operator, respectively:
By (b) of Proposition~\ref{prop:fredholm-properties} and 
(\ref{eq:essential-spectrum}),
$ \sigma_e (H + K) = \sigma_e (H)\subseteq\sigma(H) $. Because $ H $ is 
hyperbolic, $ \sigma(H) $ does not meet the imaginary axis, so neither does
$ \sigma_e (H) $. Therefore, $ H + K $ is essentially hyperbolic.
Thus, by (\ref{eq:structure}) and the remarks after it, we have proved the 
following 
\begin{theorem}
\label{thm:A+K}
An operator is essentially hyperbolic if and only if it is a compact 
perturbation of a hyperbolic operator.
\end{theorem}
we denote by $ e\mathcal{H}(E) $ the set of essentially hyperbolic
operators endowed with the topology induced by the operator norm. 
\begin{prop}
\label{prop:product}
$ e\mathcal{H}(E) $ is an open subset of $ \mathcal{L}(E) $,
and homeomorphic to the product 
$ \mathcal{H}(\mathcal{C}(E))\times\mathcal{L}_c (E) $.
\end{prop}
\begin{proof}
By Definition~\ref{defn:essentially-hyperbolic}, 
\begin{equation}
\label{eq:eHE-open}
e\mathcal{H}(E) = p\sp{-1} (\mathcal{H}(\mathcal{C})),
\end{equation}
where $ p $ is the quotient projection defined in (\ref{eq:quotient}).
Because the right term is an open subset of $ \mathcal{L}(E) $, so is
$ e\mathcal{H}(E) $. Because $ p $ is linear, continuous and surjective, 
there exists 
$ s\colon\mathcal{C}(E)\rightarrow\mathcal{L}(E) $ continuous such that
$ p\circ s  = id $. This follows from \cite[Proposition~A.1]{AM09}.
We define the continuous maps
\begin{align*}
f:e\mathcal{H}(E)\rightarrow\mathcal{H}(\mathcal{C}(E))\times\mathcal{L}_c (E),
\ \ \ 
&A\mapsto (\prc(A), A - s(\prc(A))), \\
g:\mathcal{H}(\mathcal{C}(E))\times\mathcal{L}_c (E)\rightarrow 
e\mathcal{H}(E), \ \ \ 
&(x,K)\mapsto s(x) + K.
\end{align*}
Both are well defined. In fact, by (\ref{eq:eHE-open}), 
$ \prc(A)\in\mathcal{H}(\mathcal{C}(E)) $. By the property of $ s $, 
\[
\prc(A - s(\prc(A))) = 0,
\]
thus the $ \mathcal{L}(E) $ component of $ f $ is compact, because 
$ \mathcal{L}_c (E) = \ker(\prc) $. As for $ g $,
because $ p(s(x)) = x\in\mathcal{H}(\mathcal{C}(E)) $, by (\ref{eq:eHE-open}),
$ s(x)\in e\mathcal{H}(E) $, hence
\[
\sigma_e (s(x))\cap\J = \emptyset.
\]
Because $ \sigma_e (s(x) + K) = \sigma_e (s(x)) $, 
$ s(x) + K\in e\mathcal{H}(E) $. We conclude the proof by checking that 
$ f $ and $ g $ are the inverses of each other:
\begin{gather*}
f\circ g(x,K) = f(s(x) + K) = (\prc(s(x) + K), s(x) + K - s(\prc(s(x) + K))) 
= (x,K)\\
g\circ f(A) = s(p(A)) + A - s(p(A)) = A.
\end{gather*}
\end{proof}
\begin{defn}
Let $ x\in\mathcal{A} $ be such that 
$ \sigma\sp + (x) = \sigma(x)\cap\{\re(z) > 0\} $ is open and closed in 
$ \sigma(x) $. We denote by $ p^+ (x) $ the projector 
$ P(x;\{\re(z) > 0\}) $. Similarly, we define $ p\sp - (x) $.
\end{defn}
\begin{prop}
\label{prop:H-P}
The map 
$ p\sp + \colon\mathcal{H}(\mathcal{A})\rightarrow\mathcal{P}(\mathcal{A})$ 
defines a homotopy equivalence, a homotopy inverse being the map
$ j\colon\mathcal{P}(\mathcal{A})\rightarrow\mathcal{H}(\mathcal{A}) $, 
$ j(p) = 2p - 1 $.
\end{prop}
\begin{proof}
For $ x\in\mathcal{H}(A) $, $ \sigma\sp + (x) $ is open and closed in
$ \sigma(x) $. Thus, there exists a rectangle
$ Q := (-a,a)\times (-b,b) $ such that $ \sigma(x)\subset Q $.
\begin{TAKEOUT} 
$ Q = (0,a)\times (-b,b) $ with the following properties
\begin{gather*}
\sigma\sp + (x)\subset Q\\
\sigma(x)\subset\{|\im(z)| < b\}\cap(\p Q)\sp c =: \Omega.
\end{gather*}
\end{TAKEOUT}
There is a continuous, closed and simple (in the
sense of Definiton~\ref{defn:simple}) curve $ c $, such 
that $ \text{im}(c) = \p Q\sp+ $. Thus, $ p\sp + (x) = P_c (x) $.
By (iii) of Proposition~\ref{prop:spectrum-basics}, there exists 
$ \delta > 0 $ such that, if $ d(y,x) < \delta $, then 
$ y\in\mathcal{H}(\mathcal{A}) $ and $ \sigma(y)\subset Q $. Thus
$ Q\sp+\cap\sigma(y) = \sigma\sp + (y) $ and $ p\sp + (y) = P_c (y) $.
By Theorem \ref{thm:spectrum-projector}, $ P_c $ is continuous on 
$ B(x,\delta) $. Thus, $ p\sp + $ is continuous on a neighbourhood of $ x $, 
namely $ B(x,\delta) $. Repeating the same argument for every $ x $, we
obtain that $ p\sp + $ is continuous on $ \mathcal{H}(\mathcal{A}) $.

Given a projector $ p $, $ j(p) $ is a square root of 
the unit. In fact,
\[
j(p)\sp 2 = (2p - 1)\sp 2 = 4p\sp 2 - 4p + 1 = 1.
\]
Thus $ \sigma(j(p))\subseteq\{-1,1\} $, hence $ j(p) $ is
hyperbolic. Given $ \zeta\in\mathbb{C}\setminus\{-1,1\} $, we have
\[
(j(p) - \zeta)^{-1} = \frac{\zeta}{1 - \zeta^2} + \frac{j(p)}{1 - \zeta^2} 
= \frac{1}{2}\left(-\frac{1}{\zeta + 1} + \frac{1}{1 - \zeta}\right) -
\frac{1}{2}\left(-\frac{1}{\zeta + 1} - \frac{1}{1 - \zeta}\right) j(p).
\]
Let $ c $ be a simple curve as $ c(t) = 1 + e\sp{-2\pi i t}/2 $. Thus,
following the notations of Definition~\ref{defn:simple}, we have
\[
\sigma\sp + (j(p)) = \sigma(j(p))\cap\Omega_1 (c).
\]
Therefore, we can compute the spectral projector relative to 
$ 1\in\sigma(j(p)) $ as in Theorem~\ref{thm:spectrum-projector}.
By integrating both sides of the above equality,
\[
\begin{split}
p\sp + (j(p)) &= \frac{1}{2\pi i}\int_c (j(p) - \zeta)^{-1} d\zeta\\
&= \frac{1}{2} \big(-\ind_c (-1) + \ind_c (1)\big) - 
\frac{1}{2}\big(-\ind_c (-1) - \ind_c (1)\big)j(p) \\
&= \frac{1}{2}\big(0 + 1 + j(p)\big) = \frac{1}{2}(1 + 2p - 1) = p.
\end{split}
\]
The computation above shows that $ p^+ \circ j $ is the identity map on 
$ \mathcal{P}(\mathcal{A}) $.
To prove that $ j\circ p^+ $ is homotopically equivalent to the 
identity on $ \mathcal{H}(\mathcal{A}) $, we define
\begin{equation}
\label{eq:jp+=id}
H(t,x) = \big((1 - t) x + t\big) p^+ (x) + \big((1 - t) x - t\big) p^- (x).
\end{equation}
Because $ x $ is hyperbolic, 
$ \sigma\sp + (x)\cup \sigma\sp - (x) = \sigma(x) $,
thus, by (iii) of Theorem~\ref{thm:spectrum-projector},
\begin{equation}
\label{eq:decomposition}
p\sp + (x) + p\sp - (x) = 1.
\end{equation}
By (\ref{eq:sp}) and by (ii) of Theorem~\ref{thm:spectrum-projector}, we have
\begin{equation}
\label{spectrum}
\begin{split}
\sigma(H(t,x)) &=
\sigma_{p\sp + (x)} \big((1 - t) xp^+ (x) + tp^+ (x)\big)
\cup\sigma_{p\sp - (x)}\big((1 - t) xp^- (x) - tp^- (x) \big) \\
&= \{(1 - t)\sigma^+ (x) + t\} 
\cup\{ (1 - t)\sigma^- (x) - t\}.
\end{split}
\end{equation}
The second equality follows from (i) of Theorem~\ref{thm:spectrum-projector} 
applied to $ p\sp + (x) $ (resp. $ p\sp - (x) $) and $ \sigma\sp + (x) $
(resp. $ \sigma\sp - (x) $).
Because the subsets of the complex plane $ \{\re(z) > 0\} $ 
and $ \{\re(z) < 0\} $ are convex, the sets in the
second line of (\ref{spectrum}) do not meet the imaginary axis, and thus 
$ H(t,x) $ is hyperbolic. Moreover, 
\begin{gather*}
H(0,x) = p\sp + (x) + p\sp - (x) = 1\\
H(1,x) = p\sp + (x) - p\sp - (x) = 2p\sp + (x) - 1 = j(p\sp + (x))
\end{gather*}
by (\ref{eq:decomposition}). Hence $ H $ is a homotopy of $ j\circ p\sp + $
with the identity map.
\end{proof}
Because $ \mathcal{L}_c (E) $ is a vector space, thus it is 
contractible to a point, the projection onto the first factor in 
$ \mathcal{H}(\mathcal{C}(E))\times\mathcal{L}_c (E) $ is 
a homotopy equivalence. Together with the last two propositions, we have 
proved the following
\begin{cor}
\label{cor:homotopy-equivalence}
The map $ \Psi\colon e\mathcal{H}(E)\rightarrow\mathcal{P}(\mathcal{C}(E)) $, 
$ A\mapsto p^+ (A + \mathcal{L}_c (E)) $ is a homotopy equivalence.
\end{cor}
Given an essentially hyperbolic operator $ A $,
we denote by $ P\sp + (A) $ and $ P\sp - (A) $ the spectral projectors
relative to $ \{\re(z) > 0\} $ and 
$ \{\re(z) < 0\} $, respectively.
\begin{prop}
\label{prop:componentschar}
Given a connected component $ \ccH\subset e\mathcal{H}(E) $, there exists
$ P\in\mathcal{P}(\mathcal{L}(E)) $ such that $ 2P - I\in\ccH $. Moreover,
two essentially hyperbolic operators $ A,B $ belong to the same connected
component $ \ccH $, if and only if there exists $ T\in GL_I (E) $ such that
\[
T P\sp + (A) T^{-1} - P\sp + (B)\in\mathcal{L}_c (E).
\]
\end{prop}
\begin{proof}
By Proposition~\ref{prop:product}, $ e\mathcal{H}(E) $ is an open subset of 
$ \mathcal{L}(E) $. Thus, $ \ccH $ is path-connected. 
Let $ A\in\ccH\subset e\mathcal{H}(E) $ be an essentially hyperbolic operator.
By Theorem~\ref{thm:A+K}, there exists a hyperbolic operator $ H $ such that 
\[
A - H\in\mathcal{L}_c (E).
\]
By the converse of the same theorem, the continuous convex combination
\[
\gamma\colon t\mapsto A + t(H - A)
\]
lies in $ e\mathcal{H}(E) $. Because $ \gamma(0) = A\in\ccH $, 
$ \gamma(1) = H\in\ccH $. By Proposition~\ref{prop:H-P}, 
$ j\circ p\sp + (H) = 2 P\sp + (H) - I $ is path-connected to $ H $, a path
being defined as in (\ref{eq:jp+=id}). Thus $ 2P\sp + (H) - I\in\ccH $,
and this concludes the first part of the proposition.

Given $ A,B\in\ccH $, there exists a path $ A_t $ such that
$ A(0) = A $ and $ A(1) = B $. Thus, the path 
\[
\alpha := \Psi\circ A\in\mathcal{P}(\mathcal{C}(E))
\]
connects $ \alpha(0) = \Psi(A) $ to $ \alpha(1) = \Psi(B) $. By
Proposition \ref{prop:loc-sect-Calkin}, the fibration 
$ (\mathcal{P}(\mathcal{L}(E),\mathcal{P}(\mathcal{C}(E),p) $ satisfies
the homotopy lifting property w.r.t. to the unit interval $ [0,1] $.
Thus, there exists a path of projectors $ P $ such that
\begin{equation}
\label{eq:lift}
P(0) = P\sp + (A),\ \ p(P(t)) = \alpha(t).
\end{equation}
By Theorem~\ref{thm:pair-of-projectors}, there exists $ T\in GL_I (E) $
such that
\begin{equation}
\label{eq:conjugation}
TP\sp + (A) T\sp {-1} = P(1).
\end{equation}
From (\ref{eq:lift}) with $ t = 1 $, we obtain
\[
p(P(1)) = \Psi(B) = p\sp + (B + \mathcal{L}_c (E)) = 
P\sp + (B) + \mathcal{L}_c (E) = p(P\sp + (B)).
\]
The second equality follows from the definition $ \Psi $ and the third
one from (iv) of Proposition \ref{thm:spectrum-projector}.
Thus, comparing the first and the last terms in the chain of equalities
above, we obtain
\[
P(1) - P\sp + (B) \in\mathcal{L}_c (E).
\]
Hence, by (\ref{eq:conjugation}),
\[
TP\sp + (A) T\sp {-1} - P\sp + (B)\in\mathcal{L}_c (E).
\]
\end{proof}
\section{The spectral flow}
\label{Three}
Let $ A\colon [0,1]\ra e\mathcal{H}(E) $ be a continuous path. 
By Proposition~\ref{cor:homotopy-equivalence}, $ \Psi(A(t)) $ is
a continuous path in $ \mathcal{P}(\mathcal{C}(E)) $. This path can
be lifted to a path of projectors $ P $, such that 
\[
\prc(P(t)) = \Psi(A(t)) = \prc(P\sp + (A(t))) 
\]
by Proposition~\ref{prop:loc-sect-Calkin}. We define the integer
\begin{equation}
\label{eq:def-Gar10}
\Sf(A;P) := [P(0) - P\sp + (A(0))] - [P(1) - P\sp + (A(1))].
\end{equation}
\begin{prop}
\label{prop:sf-well-defined}
The integer $ \Sf(A;P) $ does not depend on the choice of the path
of projectors $ P $.
\end{prop}
\begin{proof}
Let $ Q $ be a path of projectors such that $ \prc(Q(t)) = \prc(P(t)) $.
Thus, $ Q(0) - P(0) $ and $ Q(1) - P(1) $ are compact operators. By (ii) of 
Theorem~\ref{thm:relative-dimension}, we have
\[
\begin{split}
\Sf(A;Q) &= [Q(0) - P\sp + (A(0))] - [Q(1) - P\sp + (A(1))]
\\
&= [Q(0) - P(0)] + [P(0) - P\sp + (A(0))] \\
&- [Q(1) - P(1)] - [P(1) - P\sp + (A(1))] = \Sf(A;Q)
\end{split}
\]
By (iii) of Theorem~\ref{thm:relative-dimension}, $ [Q(t) - P(t)] $ is
constant. Thus, the third equality follows.
\end{proof}
\begin{defn}
Given $ A\colon [0,1]\rightarrow e\mathcal{H}(E) $ continuous, we define the 
\textsl{spectral flow} as
the integer $ \Sf(A;P) $ where $ P $ is any of the paths of projectors
such that $ \prc(P(t)) = \prc(P\sp + (A(t))) $. We denote it by $ \Sf(A) $.
\end{defn}
Given $ T\in\mathcal{L}(E) $ and $ S\in\mathcal{L}(F) $, we refer to
$ T\oplus S $ as the linear operator on $ E\oplus F $ such that
$ T\oplus S (x,y) = (Tx,Sy) $.
Given two paths $ A $ and $ B $ such that $ A(1) = B(0) $, we denote by 
$ A * B $ the continuous path
\[
A*B(t) = 
\begin{cases}
A(2t) & 0\leq t\leq 1/2\\
B(2t - 1) & 1/2\leq t\leq 1
\end{cases}
\]
\begin{prop}
\label{prop:sf-properties}
The spectral flow satisfies the following properties:
\begin{enumerate}
\item Given two paths $ A $ and $ B $ such that $ A(1) = B(0) $, 
$ \Sf(A*B) = \Sf(A) + \Sf(B) $;
\item the spectral flow of a constant path or a path in 
$ \mathcal{H}(\mathcal{L}(E)) $
is zero;
\item it is invariant for homotopies with endpoints in 
$ \mathcal{H}(\mathcal{L}(E)) $ and for fixed-endpoint homotopies in 
$ e\mathcal{H}(E) $;
\item if $ A_i \in C([0,1],e\mathcal{H}(E_i)) $ for $ 1\leq i\leq n $,
then $ \Sf(\oplus_{i = 1} \sp n A_i) = \sum_{i = 1} \sp n \Sf(A_i) $;
\item if $ E $ is an $ n $-dimensional linear space, then for
every integer $ - n\leq k\leq n $, there is a path such that $ \Sf(A) = k $;
\item if $ E $ has infinite dimension, then for every $ k $ there is 
$ A $ such that $ \Sf(A) = k $.
\end{enumerate}
\end{prop}
\begin{proof}
(i). Let $ A,B $ be two paths such that $ A(1) = B(0) $. We can choose
paths of projectors $ P $ and $ Q $ such that 
$ \prc(P(t)) = \prc(P\sp + (A(t)) $ and $ \prc(Q(t)) = \prc(P\sp + (B(t))) $, 
with $ Q(0) = P(1) $. 
Denote by $ C $ and $ R $ 
the paths $ A*B $ and $ P*Q $, respectively. Then,
\[
\begin{split}
\Sf(A*B) &= [R(0) - P\sp + (C(0))]
- [R(1) - P\sp + (C(1))]\\
&= [P(0) - P\sp + (A(0))] -
[Q(1) - P\sp + (B(1))] = [P(0) - P\sp + (A(0))] \\
&- [P(1) - P\sp + (A(1))]
+ [Q(0) - P\sp + (B(0))] - [Q(1) - P\sp + (B(1))] \\
&= \Sf(A) + \Sf(B).
\end{split}
\]\par
\noindent (ii). If $ A $ is constant, $ P\sp + (A(t)) $ is constant; if
$ A $ is hyperbolic, $ P\sp + (A(t)) $ is continuous. In both cases,
$ P\sp + (A(t)) $ is a continuous path and can be chosen as a lifting path
of $ \prc(P\sp + (A(t))) $. Therefore,
\[
\Sf(A) = [P\sp + (A(0)) - P\sp + (A(0))] - [P\sp + (A(1)) - P\sp + (A(1))] = 0.
\]
\noindent (iii). Let $ H\colon I\times I \ra e\mathcal{H}(E) $ be a continuous
map. By the homotopy lifting property of the fibre bundle
$ \prc\colon\mathcal{P}(\mathcal{L}(E))\ra\mathcal{P}(\mathcal{C}(E)) $ w.r.t.
$ I\sp 2 $, there exists $ P\colon I\times I\ra \mathcal{P}(\mathcal{L}(E)) $ 
such that
\[
P(t,s) - P\sp + (H(t,s))\in\mathcal{L}_c (E),\text{ for every } t,s.
\]
Let $ H(\cdot,0) = A $ and $ H(\cdot,1) = B $. We have
\[
\Sf(A) = [P(0,0) - P\sp + (H(0,0))] - [P(1,0) - P\sp + (H(1,0))].
\]
For $ i=0,1 $ and every $ s $, the operator $ P(i,s) - P\sp + (H(i,s)) $ 
is compact. For a fixed $ i $, the right summand is constant or continuous, 
whether the homotopy has fixed endpoints in $ e\mathcal{H}(E) $ or lies in 
$ \mathcal{H}(\mathcal{L}(E)) $. In both cases, is continuous. By (iii) of
Theorem~\ref{thm:relative-dimension}, there are integers $ k_1,k_2 $ such that
\[
[P(i,s) - P\sp + (H(i,s))] = k_i
\]
for every $ 0\leq s\leq 1 $ and $ i=0,1 $. Thus, 
$ \Sf(A) = k_0 - k_1 = \Sf(B) $.\par
\noindent (iv). Let $ P_i $ be continuous paths of projectors such that 
$ P_i (t) - P\sp + (A_i (t))\in\mathcal{L}_c (E_i) $.
\[
\begin{split}
\Sf(\oplus_{i = 1}\sp n A_i) =& \left[\oplus_{i = 1} \sp n P_i (0) - 
\oplus_{i = 1} \sp n P\sp + (A_i (0))\right] - 
\left[\oplus_{i = 1} \sp n P_i (1) - 
\oplus_{i = 1} \sp n P\sp + (A_i (1))\right] \\
=& \sum_{i = 1}\sp n [P_i (0) - P\sp + (A_i (0))] - 
[P_i (1) - P\sp + (A_i (1))] = \sum_{i = 1} \sp n \Sf(A_i).
\end{split}
\]
\par
\noindent (v). We denote the identity on $ \R\sp k $ by $ I_k $.
Given $ 0\leq k\leq n $, the spectral flow of 
\[
A(t) = (2t - 1) I_k \oplus I_{n - k}
\]
can be computed using $ P(t)\equiv I_n $. Because $ P\sp + (A(1)) = I_n $
and $ P\sp + (A(0)) = 0\oplus I_{n - k} $, we have
\[
\Sf(A;I_n) = [I_n - 0\oplus I_{n - k}] - [I_n - I_n] = k
\]
by (i) of Theorem~\ref{thm:relative-dimension}. We define
$ \overline{A}(t) := A(1 - t) $. By property (i), proved above,
$ \Sf(\overline{A};I_n) = -k $.\par
\noindent (vi). Given $ k\in\mb{Z} $, let $ E = X\sp k \oplus R_k $ 
where $ X\sp k $ is a closed subspace and $ \dim(R_k) = k $. Thus, the 
spectral flow of 
$ A(t) = (2t - 1) I_{R_k} \oplus I_{X\sp k} $ can be computed
with $ P(t)\equiv I $. We obtain $ \Sf(A;I) = k $ and 
$ \Sf(\overline{A};I) = -k $.
\end{proof}
\subsection{Spectral sections}
The definition of spectral flow we used corresponds to the one given 
by C.~Zhu and Y.~Long in \cite{ZL99} for paths of \textsl{admissible}
operators (see \cite[Definition~2.3]{ZL99}), 
which are essentially hyperbolic. We recall the definition of
s-section:
\begin{defn}
\label{defn:section}
An \textsl{s-section} for a path of projectors $ Q $ on 
$ J\subset [0,1] $ is a continuous path $ P $ such that 
$ P(t) - Q(t)\in\mathcal{L}_c (E) $.
\end{defn}
Given a continuous path $ A\colon [0,1]\ra e\mathcal{H}(E) $,
the authors show in \cite[Lemma~2.5]{ZL99} and \cite[Corollary~2.1]{ZL99} that 
there exists a partition of the unit interval $ (J_k)_{k = 1} \sp n $
and $ P_k \colon J_k\ra\mathcal{P}(\mathcal{L}(E)) $ such that
\begin{gather}
\label{eq:ZL99-1}
P_k \text{ is an s-section for } P\sp + (A)\text{ on } J_k\\
\label{eq:ZL99-2}
P_k (t)\text{ is a spectral projector of } A(t).
\end{gather}
Then, they define
\begin{equation}
\label{eq:def-ZL99}
\Sf(A) = \sum_{k = 1} \sp n \Sf(A_k;P_k)
\end{equation}
where $ A_k $ is the restriction of $ A $ to $ J_k $. In our definition
we do not need to partition the unit interval because we dropped the
requirement (\ref{eq:ZL99-2}).
This allows us to simplify the definition of spectral flow and to provide 
simpler proofs of well known properties - such as the homotopy invariance - 
than the original ones in \cite[Proposition~2.2]{ZL99} or in \cite{Phi96}.

We conclude by showing that there exists a path $ A $ such that 
$ P\sp + (A(t)) $ does not admit an s-section fulfilling (\ref{eq:ZL99-2}).
\begin{example}
Consider the decomposition $ E = R_1 \oplus X_{-} \oplus X_+ $, where
$ E $ is a Banach space, $ X_- $ and $ X_+ $ are closed, infinite-dimensional
subspaces and $ \dim(R_1) = 1 $. Denote by $ P_1,P_-,P_+ $ the projectors 
onto $ R_1,X_- $ and $ X_+ $, respectively. Define
\[
A\colon [0,1]\ra e\mathcal{H}(E),\quad A(t) = P_+ - P_- + (2t - 1) P_1.
\]
Then $ A(t)\in e\mathcal{H}(E) $, and no continuous 
s-section satisfying (\ref{eq:ZL99-2}) exists.
\end{example}
\begin{proof}
For every $ t\in [0,1] $ we can write 
$ A(t) = P_+ - (P_- + P_1) + 2t P_1 $; because 
$ P_+ - P_- - P_1 \in\mathcal{H}(\mathcal{L}(E)) $ (in fact, is
a square root of the identity) and $ P_1 $ is compact,
$ A(t)\in e\mathcal{H}(E) $. By contradiction, suppose that such $ P $ 
exists. We have
\[
A(0) = P_+ - P_- - P_1,\ \ 
P\sp + (A(0)) = P_+,\ \ \sigma(A(0)) = \{-1,1\}.
\]
Because $ P(0) $ is spectral, there exists $ \Sigma_0\subset\sigma(A(0)) $
such that $ P(0) = P(A(0);\Sigma_0) $. The only choice is $ \Sigma_0 = \{1\} $,
thus $ P(0) = P(A(0);\{1\}) = P_+ $. On the other endpoint,
\[
A(1) = P_+ - P_- + P_1,\ \ 
P\sp + (A(1)) = P_+ + P_1,\ \ \sigma(A(1)) = \{-1,1\}.
\]
As above, $ P(1) $ is spectral and $ P(1) = P(A(1);\{1\}) = P_+ + P_1 $.
Because $ P $ is an s-section and $ P\sp + (A(t)) - P\sp + (A(s)) $ is
compact for every $ 0\leq t,s\leq 1 $, $ P(t) - P(s) $ is also compact.
By (iii) of Theorem~\ref{thm:relative-dimension}, $ m(t) := [P(t) - P(0)] $ 
is constant. Because $ m(0) = 0 $, $ m(1) = 0 $. But,
\[
0 = m(1) = [P(1) - P(0)] = [(P_+ + P_1) - P_+] = [P_1] = 1
\]
where the last equality follows from (i) of 
Theorem~\ref{thm:relative-dimension}. Thus, we obtained a contradiction.
\end{proof}
\section{Spectral flow as group homomorphism}
\label{Four}
By (iii) and (i) of Proposition \ref{prop:sf-properties}, the spectral
flow determines a $ \mb{Z} $-valued group homomorphism on the fundamental
group of each connected component of $ e\mathcal{H}(E) $. Given
a projector $ P $, we denote by $ \Sf_P $ the spectral flow on the
fundamental group of the connected component of $ 2P - I $.\vskip .4em
The fiber of $ \prc\colon \mathcal{P}(\mathcal{L}(E))\rightarrow
\mathcal{P}(\mathcal{C}(E)) $ over a point of the base space, 
$ P + \mathcal{L}_c (E) $ is the set
\begin{gather*}
\mathcal{P}_c (E;P) = \{Q\in\mathcal{P}(\mathcal{L}(E)):
Q - P\in\mathcal{L}_c (E)\}\\
i\colon\mathcal{P}_c (E;P)\hookrightarrow \mathcal{P}(\mathcal{L}(E)).
\end{gather*}
\begin{prop}
\label{prop:PcE}
For every projector $ P $, the connected components of $ \mathcal{P}_c (E;P) $
correspond to $ \mb{Z} $ through the bijection $ Q\mapsto [P - Q] $.
Moreover, if the range and the kernel have infinite dimension,
$ \pi_1 (\mathcal{P}_c (E;P),P)\cong\mb{Z}_2 $.
\end{prop}
The two facts follow from \cite[Theorem~6.3]{Gar08} and
\cite[Theorems~7.2,7.3]{Gar08}.\vskip .2em

Because $ \prc $ induces a Serre fibration, the sequence of homomorphisms
\begin{equation}
\label{eq:exact-1}
\xymatrix{%
\pi_1 (\mathcal{P}_c (E;P),P) \ar[r]^-{i_*} &
\pi_1 (\mathcal{P}(\mathcal{L}(E)),P) \ar[r]^-{\prc_*} &
\pi_1 (\mathcal{P}(\mathcal{C}(E)),P + \mathcal{L}_c (E))
}
\end{equation}
is exact. The homotopy equivalence $ \Psi $ defined in Corollary 
\ref{cor:homotopy-equivalence} determines a group isomorphism
\[
\Psi_* \colon\pi_1 (e\mathcal{H}(E),2P - I)\ra 
\pi_1 (\mathcal{P}(\mathcal{C}(E)),P + \mathcal{L}_c (E)).
\]
\begin{TAKEOUT}
We set $ \cnn = \Sf_P\circ (\Psi_*) \sp{-1} $.
\end{TAKEOUT}
\begin{theorem}
There exists a homomorphism $ \cnn $ such that
the sequence
\begin{equation}
\label{eq:sf=index}
\xymatrix@C+0.4cm{%
\pi_1 (\mathcal{P}(\mathcal{L}(E)),P) \ar[r]^-{\prc_*} &
\pi_1 (\mathcal{P}(\mathcal{C}(E)),P + \mathcal{L}_c (E)) 
\ar[r]^-{\cnn} & \mb{Z},
}
\end{equation}
is exact and $ \cnn\circ\Psi_* = \Sf_P $.
\end{theorem}
\begin{proof}
Let $ a $ be a loop at the base point $ P + \mathcal{L}_c (E) $, and
let $ Q $ be a path of projectors such that $ Q(0) = P $ and 
$ \prc(Q(t)) =  a(t) $. Because 
$ a(0) = a(1) $, $ Q(1) - P $ is compact.
We define
\[
\cnn([a]) = [Q(1) - P].
\]
Arguing as in Proposition~\ref{prop:sf-well-defined}, $ \cnn $ is
well defined. Let $ A $ be a closed path in $ e\mathcal{H}(E) $
and $ Q $ a path of projectors such that
\[
\prc(Q(t)) = \Psi(A(t)),\ \ Q(0) = P.
\]
Hence, $ Q(t) - P\sp + (A(t)) $ is compact for every $ t\in [0,1] $.
By (\ref{eq:def-Gar10}), $ \Sf(A) = [Q(1) - P] $. By the definition
above, $ \cnn([\Psi\circ A]) = [Q(1) - P] $, thus $ \cnn\circ\Psi_* = \Sf_P $.
Because $ \Psi_* $ is invertible, $ \cnn $ is a homomorphism.
We prove that $ \cnn $ is exact.

\noindent 
$ \ker(\cnn)\subseteq\text{im}(\prc_*) $: Let $ a $ be a loop at the base 
point $ P + \mathcal{L}_c (E) $ such that
$ [a]\in\ker(\cnn) $. Then, 
\[
[Q(1) - P] = 0.
\]
By Proposition \ref{prop:PcE}, $ P $ and $ Q(1) $ are in the same connected
component in $ \mathcal{P}_c (E;P) $. Thus, there exists a continuous path 
of projectors $ R $ such that
\[
R(0) = Q(1),\ \ R(1) = P,\ \ R(t) - P\in\mathcal{L}_c (E)\text{ for every } t.
\]
Set $ S := Q*R $. It is a closed path of projectors, and 
$ \prc\circ S = a * c_{\prc(P)} $, where $ c_{\prc(P)} $ is the
constant path $ \prc(P) $. Thus, $ [a]\in\mathrm{im}(\prc_*) $.

\noindent $ \text{im}(\prc_*)\subseteq\ker(\cnn) $: Given a loop 
$ P_t \in\mathcal{P}(\mathcal{L}(E)) $ at the base point $ P $, we have
\[
\cnn(\prc_* (P_t)) = [P(1) - P] = 0
\]
because $ P(1) = P $. 
\end{proof}
\begin{prop}
\label{prop:image}
Given a projector $ P $, $ m\in\Im(\Sf_P) $ if and only if there exists
a projector $ Q $ such that $ Q - P $ is compact, $ [Q - P] = m $ and is 
path-connected to $ P $.
\end{prop}
By the previous theorem, $ \Im(\Sf_P) = \Im(\cnn) $. Therefore, the
proposition follows from the definition of $ \cnn $. 

We define the following properties:
\begin{itemize}
\item[h1)] 
$ P $ is path-connected to a projector $ Q $ such that 
$ Q - P $ is compact and $ [Q - P] = m $.
\item[h2)] the image of $ \prc_*\colon
\pi_1 (\mathcal{P}(\mathcal{L}(E)),P)\ra
\pi_1 (\mathcal{P}(\mathcal{C}(E),\prc(P)) $ is trivial.
\end{itemize}
\begin{cor}
\label{cor:characterisation}
Given $ P\in\mathcal{P}(\mathcal{L}(E)) $, we characterise the
kernel and the image of the spectral flow $ \Sf_P $:
\begin{enumerate}
\item $ m\in\mathrm{im}(\Sf_P) $ if and only if $ P $ fulfills
property h1).
\item $ \mathrm{im}(\prc_*)\cong\ker(\Sf_P) $. $ \Sf_P $
is injective if and only if $ P $ fulfills property h2).
\end{enumerate}
\end{cor}
The isomorphism classes of the kernel and the image of $ \Sf_P $ depend
only on the conjugacy class of $ P + \mathcal{L}_c (E) $ in 
$ \mathcal{P}(\mathcal{C}(E)) $. We show that in many cases
we can find a projector $ P $ such that $ \Sf_P $ is an isomorphism.
\begin{lemma}
\label{lem:X+Y}
Let $ E $ be a Banach space, and $ X,Y\subset E $ closed subspaces such
that $ X\cong Y $ and $ X\oplus Y = E $. Then, the projectors
$ P_X,P_Y $ with ranges $ X $ and $ Y $ respectively, are connected by a 
continuous path in $ \mathcal{P}(\mathcal{L}(E)) $.
\end{lemma}
A proof of this can be found in \cite[\S9]{PR87} or in \cite{Mit70}.
\begin{prop}
\label{prop:surjective}
Let $ X,Y\subset E $ be as above. Suppose that
$ X $ is isomorphic to its closed subspaces of co-dimension $ m $.
Let $ P $ be the projector onto $ X $ with kernel $ Y $. Then $ P $ 
satisfies the property h1) w.r.t. $ m $.
\end{prop}
\begin{proof}
Let $ X\sp m,R_m\subset X $ be closed subspaces such that $ \dim(R_m) = m $
and $ X\sp m\cong X $. We have the following decompositions and isomorphism:
\[
E = R_m \oplus X\sp m\oplus Y,\ \ X\sp m\cong Y,\ \ R_m \oplus X\sp m = X.
\]
By applying Lemma \ref{lem:X+Y} to $ X\sp m \oplus Y $ and subspaces
$ X\sp m $ and $ Y $, we obtain that 
$ P_{X\sp m} $ is connected to $ P_Y $. By applying it a second time to $ E $ 
and subspaces $ X $ and $ Y $, we obtain that $ P_X $ is 
connected to $ P_Y $. Hence, $ P_X $ is connected to 
$ P_{X\sp m} $.
\end{proof}
In Proposition~\ref{prop:surjective}, we required $ E $ to be isomorphic to a 
cartesian product of a space $ X $ with itself, but it suffices that $ E $ 
has a complemented
subspace $ F $ fulfilling the requirements of
Proposition \ref{prop:surjective}. In fact, if $ A_t \in e\mathcal{H}(F) $
is such that $ \Sf(A_t) = m $, then $ \Sf(I\oplus A_t) = m $, by (iv) of
Proposition \ref{prop:sf-properties}.
\begin{prop}
\label{prop:injective}
Given $ P\in\mathcal{P}(\mathcal{L}(E)) $, the map
$ \pi\colon GL(E)\rightarrow \mathcal{P}(\mathcal{L}(E)), 
\ \ T\mapsto TP T^{-1} $ defines a principal bundle with fiber
$ GL(X)\times GL(Y) $, where $ X = \ran(P) $ and $ Y = \ker(P) $.
\end{prop}
A proof of this can be found in 
\cite[Theorem~2.1]{PR90} or in \cite[Proposition~1.2]{AM09}. Both theorems
are stated in the more general setting of Banach algebras.
\begin{cor}
\label{cor:injective}
If $ GL(E) $ is simply-connected and $ GL(X),GL(Y) $ are connected, then
$ \Sf_P $ is injective. 
\end{cor}
\begin{proof}
Because a locally trivial bundle is a Serre fibration, we have a long exact 
sequence of homomorphisms that ends
\begin{equation*}
\xymatrix{
\pi_1 (GL(E),I) \ar[r]^-{\pi_*} &
\pi_1 (\mathcal{P}(\mathcal{L}(E)),P) \ar[r]^-{\Delta} &
\pi_0 (GL(X)\times GL(Y),I).
}
\end{equation*}
Thus, if $ GL(E) $ is simply connected and $ GL(X) $ and $ GL(Y) $ are
connected, the middle group is trivial, hence in (\ref{eq:sf=index})
$ \prc_* $ is the trivial map, thus $ \cnn $ is injective and $ \Sf_P $
is injective. 
\end{proof}
Then, we have sufficient conditions for a Banach space to
have at least one projector $ P $ such that $ \Sf_P $ is an isomorphism.
\begin{theorem}
\label{thm:isomorphism}
Let $ E = X\oplus X $ be such that $ X $ is isomorphic to its hyperplanes, 
$ GL(E) $ is simply-connected and $ GL(X) $ is connected. Then
$ \Sf_{P_X} $ is an isomorphism.
\end{theorem}
\begin{proof}
That $ \Sf_{P_X} $ is surjective follows from Proposition 
\ref{prop:surjective}. From the corollary above, $ \Sf_{P_X} $ is also 
injective.
\end{proof}
Let us consider the particular case, where 
$ E $ is isomorphic to $ E\times E $ and to its
hyperplanes, and $ GL(E) $ is \textsl{contractible} to a point.
This, in fact, is the case of the most common 
infinite-dimensional spaces as separable Hilbert spaces, 
$ c_0,\ell\sp p $ with $ p\geq 1 $, and $ L\sp p (\Omega,\mu) $ with $ p > 1 $,
$ C(K,F) $ for large classes of compact spaces $ K $ and Banach spaces 
$ F $, and many others. For a richer list, see \textsc{Theorem}~2
of \cite{Sch98} and \cite{Kui65,Art66,Neu67,Mit70}. Sequence spaces
$ \ell\sp p,\ell\sp\infty $ and $ c_0 $ are also \textsl{prime} 
(see \cite[Theorem~2.2.4]{AK06} and \cite{Lin67}),
that is, they are isomorphic to their complemented, infinite-dimensional
subspaces. Thus, for \textsl{every} projector $ P $ such that
$ \ran(P) $ and $ \ker(P) $ have infinite dimension, $ \Sf_P $ is
an isomorphism.
\subsubsection*{Trivial spectral flow}
When $ P $ is a projector with a finite-dimensional range or kernel, 
$ P + \mathcal{L}_c (E) $ is 0 or 1, then
its connected component in $ \mathcal{P}(\mathcal{C}(E)) $ is $ \{0\} $
or $ \{1\} $. Hence, $ \Sf_P = 0 $. This is the case of finite-dimensional 
spaces. A space is said to be \textsl{undecomposable} if the only projectors
are as above. In \cite{GM93}, W.~T.~Gowers and B.~Maurey showed the 
existence of an infinite-dimensional, undecomposable space.
\subsubsection*{Non-trivial and not surjective spectral flow}
In \cite{GM97}, W.~T.~Gowers and B.~Maurey proved the existence of a space 
isomorphic to their subspaces of co-dimension two, but not their 
hyperplanes. If we denote by $ X $ a space with such a
property and by $ P $ the projector onto the first factor in 
$ E = X\oplus X $, then $ 2\in\text{im}(\Sf_P) $ by Proposition 
\ref{prop:surjective}. However, if $ 1\in\text{im}(\Sf_P) $, by 
Proposition~\ref{prop:image} and c2) in \S\ref{ssect:spaces-of-projectors}, 
$ X $ would be isomorphic to its hyperplanes.
\subsection{The Douady space}
\label{ssect:douady}
We show the existence of a projector $ P $ such that $ \cnn $, and
thus $ \Sf_P $, is not injective.
\begin{prop}
\label{prop:delta}
Let $ E = X\oplus X $. Given $ T\in GL(X) $, there exists 
a loop $ x $ in the space of projectors such that 
\[
\Delta(x) = 
\begin{pmatrix}
T & 0 \\
0 & T^{-1}
\end{pmatrix},
\]
where $ \Delta $ is the connecting homomorphism in the sequence of
Corollary~\ref{cor:injective}
\end{prop}
\begin{proof}
Let $ M $ be the operator defined in the line above. Let $ U_t \in GL(E) $ be
such that $ U(1) = M $ and $ U(0) = I $. The existence of $ U $ follows
from the fact that $ T\oplus T' $ is connected to $ TT' \oplus I $ ( 
see \cite{Mit70}). Because $ M $ commutes with $ P_X $, the path
\[
P(t) = U(t) P_X U(t) \sp{-1}
\]
is a loop in $ \mathcal{P}(\mathcal{L}(E)) $ with base point $ P_X $.
We denote its homotopy class by $ x $. The path $ U_t $ 
is a lifting path for $ P $. Hence $ \Delta(x) = U(1) = M $.
\end{proof}
Let $ F $ and $ G $ be Banach spaces such that
\begin{enumerate}
\item every bounded map $ G\rightarrow F $ is compact;
\item both $ F $ and $ G $ are isomorphic to their hyperplanes.
\end{enumerate}
The next Lemma follows from a more general result of A.~Douady, 
\cite[Proposition~1]{Dou65}. 
We briefly sketch the proof by B.~S.~Mitjagin in \cite{Mit70}.
\begin{lemma}
Let $ X = F\oplus G $, $ F $ and $ G $ as above. Then, there exists a 
continuous, surjective homomorphism $ j\colon GL(X)\rightarrow\mb{Z} $.
\end{lemma}
\begin{proof}
Let $ T\in GL(X) $ be an invertible operator and $ S $ be its inverse. We have
\[
\begin{pmatrix}
I_F & 0 \\
0 & I_G
\end{pmatrix}
= TS = 
\begin{pmatrix}
T_{11} S_{11} + T_{12} S_{21} & T_{11} S_{12} + T_{12} S_{22} \\
T_{21} S_{11} + T_{22} S_{21} & T_{22} S_{22} + T_{21} S_{12}.
\end{pmatrix} 
\]
A similar equality holds for $ ST $. Taking the first element of the
diagonals of $ TS $ and $ ST $, respectively, we have the following
relations
\[
T_{11} S_{11} + T_{12} S_{21} = I_F,\ \ S_{11} T_{11} + S_{12} T_{21} = I_F.
\]
Because $ S_{21} $ and $ T_{21} $ are compact operators, 
$ T_{11} $ and $ S_{11} $ are the essential inverse of each other.
According to (e) of Proposition~\ref{prop:fredholm-properties}, 
$ T_{11} $ is a Fredholm operator. We define
\[
j(T) = \ind(T_{11}).
\]
By (a) of Proposition~\ref{prop:fredholm-properties}, there exists 
$ \var > 0 $ such that, if
$ \no{T' _{11} - T_{11}} < \var $, then $ T' _{11} $ is a Fredholm
operator and $ \ind(T' _{11}) = \ind(T_{11}) $. This proves the continuity.
Moreover, given two invertible operators $ T $ and $ S $, we have
\[
\begin{split}
j(TS) &= \ind(TS)_{11} = \ind(T_{11} S_{11} + T_{12} S_{21}) \\
&= \ind(T_{11} S_{11}) = \ind(T_{11}) + \ind(S_{11}),
\end{split}
\]
where (b) and (c) of Proposition~\ref{prop:fredholm-properties} have been 
used. 
Thus, $ j $ is a group homomorphism.
Let $ F\sp 1 $ and $ G\sp 1 $ be 
hyperplanes of $ F $ and $ G $, respectively. We define
\begin{gather*}
\sigma\colon F\ra F\sp 1,\ \ \tau\colon G\sp 1 \ra G\\
F = \langle v\rangle\oplus F\sp 1,\ \ G = \langle w\rangle\oplus G\sp 1\\
B\colon G\ra F,\ \ tw + y\mapsto tv.
\end{gather*}
where $ \sigma,\tau $ are isomorphisms, which exist by (i). We define
\[
T(x,y) = (\sigma(x) + B(y),\tau(Py))
\]
where $ P:G\ra G $ is a projector onto $ G\sp 1 $.
$ T $ is invertible and $ \ind(T_{11}) = \ind(\sigma) = -1 $. Because
$ j $ is a homomorphism, it is surjective.
\end{proof}
\begin{prop}
If $ E = X\oplus X $, where $ X $ is a direct sum of two spaces 
$ F $ and $ G $ as in (i) and (ii) above, 
then $ \Sf_{P_X} $ is surjective, but not injective.
\end{prop}
\begin{proof}
From the lemma above, for every $ k\in\mb{Z} $, there exists $ T_k \in GL(X) $
such that $ j(T_k) = k $. By Proposition \ref{prop:delta}, there exists
$ x_k \in\pi_1 (\mathcal{P}(\mathcal{L}(E)),P_X) $ such that 
$ \Delta(x_k) = T_k\oplus T_k \sp{-1} $
and $ x_k \neq x_h $ if $ k\neq h $. Then
$ \pi_1 (\mathcal{P}(\mathcal{L}(E)),P_X) $ has infinitely many elements, while
$ \pi_1 (\mathcal{P}_c (E;P),P_X) $ is a finite group, by Proposition
\ref{prop:PcE}. Hence, in (\ref{eq:exact-1}), $ i_* $ is not surjective, thus 
$ x_k\not\in\text{im}(i_*) $ for infinitely many $ k\in\mb{Z} $. Hence,
\[
\prc_* (x_k)\neq 0,\ \ \delta_{P_X} (\prc_* (x_k)) = 0.
\]
because the sequence (\ref{eq:exact-1}) is exact; therefore, 
the kernel of $ \delta_{P_X} $ is not trivial.
Because $ E $ and $ X $ fulfill the hypothesis of 
Proposition~\ref{prop:surjective}, $ \Sf_{P_X} $ is surjective.
\end{proof}
By \cite[Theorem~4.23]{Rya02}, examples of pairs of Banach spaces 
as in (i) and (ii) are given by $ (\ell\sp p,\ell\sp 2) $, 
with $ p > 2 $.
\vskip .4em
In the next proposition, we show the existence of a projector $ P $ whose
range and kernel are isomorphic to their hyperplanes, but $ \Sf_P = 0 $.
\begin{prop}
Let $ X = F \oplus G $, where $ F $ and $ G $ fulfill properties (i) and
(ii). Then, $ \Sf_{P_F} = 0 $.
\end{prop}
\begin{proof}
Let $ 0\geq -m\in\text{im}(\Sf_{P_F}) $. Then, by (i) of
Corollary~\ref{cor:characterisation}, $ P_F $ 
is connected to a projector $ Q\in\mathcal{P}(\mathcal{L}(X)) $ 
such that $ Q - P_F $ is compact and $ [P_F - Q] = m $.
Let $ P_m\in\mathcal{P}(\mathcal{L}(X)) $ be a projector 
onto a subspace $ F\sp m \subset F $ of co-dimension $ m $ such that
$ P_m (I - P_F) = 0 $. Thus, $ P_F - P_m $ is compact and 
$ [P_F - P_m] = m $. 
Therefore, $ Q - P_m $ is compact and $ [P_m - Q] = 0 $.
By Proposition~\ref{prop:PcE}, $ Q $ is connected to $ P_m $.
Hence, $ P_F $ is connected to $ P_m $.
By Theorem \ref{thm:pair-of-projectors}, there exists a continuous
path $ U_t \in GL(X) $ such that
\[
U(0) = I,\ \ U(1)P_F = P_m U(1).
\]
From these relations, it follows that $ U(1)_{11} (F) = F\sp m $, hence 
$ j(U(1)) = -m $, and $ j(U(0)) = 0 $. By the lemma above, $ j(U(t)) $ is 
constant, therefore $ m = 0 $.
\end{proof}
\section{The Fredholm index of $ F_A $ and the spectral flow}
\label{Five}
\begin{defn}
A path $ A\colon\R\ra\mathcal{L}(E) $ of bounded operators is said to be
\textsl{asymptotically hyperbolic} if the limits $ A(+\infty) $
and $ A(-\infty) $ exist and are hyperbolic operators.
\end{defn}
If $ A $ is also essentially hyperbolic, we can define the spectral flow as 
follows: Because the set of hyperbolic operators is an open subset of 
$ \mathcal{L}(E) $, there exists 
$ \delta > 0 $ such that $ A(t) $ is hyperbolic for every 
$ t\in (-\infty,-\delta]\cup[\delta,+\infty) $. We set
\begin{equation}
\label{sf_hyp}
\Sf(A) = \Sf(A,[-\delta,\delta]).
\end{equation}
The definition does not depend on the choice of $ \delta $ by (i,ii) of
Proposition \ref{prop:sf-properties}.
Let $ P\sp + (A) $ be the spectral projector $ P(A;\{\re(z) > 0\}) $. 
\begin{defn}
An asymptotically hyperbolic path is called \emph{essentially splitting} 
if the following properties
\begin{enumerate}
\item $ P^+ (A(+\infty)) - P^+ (A(-\infty)) $ is compact;
\item $ [A(t),P^+ (A(+\infty))] $ is compact for every $ t\in\R $,
\end{enumerate}
hold, where $ [A,P] = A P - P A $.
\end{defn}
The definition above corresponds to the one given in 
\cite[Theorem~6.3]{AM03} when (i) holds.
For short, we will refer to essentially splitting as a path satisfying
(i) and (ii).
\begin{lemma}
\label{essential_splitting}
Let $ A $ be an asymptotically hyperbolic and essentially hyperbolic path. 
It is also essentially splitting if and only if 
$ P\sp + (A(t)) - P\sp + (A(s)) $ is compact for every $ t,s\in\R $.
\end{lemma}
\begin{proof}
Suppose $ A $ is essentially splitting.
we denote by $ E^+ $ and $ E^- $ the ranges 
of $ P^+ (A(+\infty)) $ and $ P^- (A(+\infty)) $, respectively. 
With respect to the decomposition $ E = E^+ \oplus E^- $, we can write
$ A(t) $ block-wise:
\begin{equation*}
A(t) = 
\begin{pmatrix}
A_+ (t) & K_{\pm} (t) \\
K_{\mp} (t) & A_{-} (t)
\end{pmatrix}
\end{equation*}
where $ K_{\pm} (t) $ and $ K_{\mp} (t) $ are compact operators by (ii).
Then $ A(t) $ is a compact perturbation of $ A_+ (t) \oplus A_- (t) $.
Therefore
\begin{equation}
\label{eq:essential-splitting-2}
P^+ (A(t)) - P^+ (A_+ (t)) \oplus P^+ (A_- (t))\in\mathcal{L}_c (E)
\text{ for every }t\in\R.
\end{equation}
Because $ A(+\infty) $ commutes with $ P\sp + (A(+\infty)) $ and 
$ P\sp - (A(+\infty)) $,
$ K_{\pm} (+\infty),K_{\mp} (+\infty) $ are null operators, hence
$ A(+\infty) = A_+ (+\infty)\oplus A_- (+\infty) $.
By the definition of essential spectrum and (b) of 
Theorem~\ref{prop:fredholm-properties}, 
$ \sigma_e (A(t)) = \sigma_e (A_+ (t))\cup\sigma_e (A_- (t)) $. Therefore,
\begin{gather}
\label{eq:essential-splitting-21}
P\sp + (A_+ (+\infty)) = I_{E\sp +},\
A_+ (t)\in e\mathcal{H}(E\sp +)\\
\label{eq:essential-splitting-22}
P\sp + (A_- (+\infty)) = 0_{E\sp-},\
A_- (t)\in e\mathcal{H}(E\sp -).
\end{gather}
From the first part of (\ref{eq:essential-splitting-21}) and 
Proposition~\ref{prop:componentschar},
\[
A_+ (+\infty)\in\mathcal{X}_{I_{E\sp+}}\subset e\mathcal{H}(E\sp +).
\]
where $ \mathcal{X}_{I_{E\sp +}} $ is the connected component of $ I_{E\sp +} $
in $ e\mathcal{H}(E\sp +) $. Because $ e\mathcal{H}(E\sp+) $ is locally
path-connected, $ \mathcal{X}_{I_{E\sp +}} $ is open. Then,
by the second part of (\ref{eq:essential-splitting-21})
\[
A_+ (t)\in\mathcal{X}_{I_{E\sp+}}\text{ for every } t\in\R.
\]
By Proposition~\ref{prop:componentschar},
\begin{equation}
\label{eq:essential-splitting-3}
P\sp + (A_+ (t)) - I_{E\sp +}\in\mathcal{L}_c (E)\text{ for every } t\in\R.
\end{equation}
Similarly, from (\ref{eq:essential-splitting-22}) and
Proposition~\ref{prop:componentschar}, we obtain
\begin{equation}
\label{eq:essential-splitting-4}
P\sp + (A_- (t))\in\mathcal{L}_c (E)\text{ for every } t\in\R.
\end{equation}
By the definition of $ P\sp+ (A(+\infty)) $, 
$ I_{E^+} \oplus 0_{E^-} = P^+ (A(+\infty)) $. Thus, by
(\ref{eq:essential-splitting-2},\ref{eq:essential-splitting-3}) and
(\ref{eq:essential-splitting-4}), 
\[
P\sp + (A(t)) - P^+ (A(+\infty))\in\mathcal{L}_c (E)\text{ for every }t\in\R.
\]
Conversely, suppose that each of the projectors of the
set $ \{P^+ (A(t)):t\in\mathbb{R}\} $ is a compact perturbation of the
others. Let $ a > 0 $ be such that $ A(s) $ is hyperbolic for every 
$ |s| > a $. By the continuity of $ P\sp + $, it follows that
\[
P^+ (A(+\infty)) - P^+ (A(-\infty)) = 
\lim_{\genfrac{}{}{0pt}{}{s\rightarrow+\infty}{|s| > a}} 
\Big(P^+ (A(s)) - P^+ (A(-s))\Big).
\]
because $ A $ is asymptotically hyperbolic. Hence, the left member is the 
limit of a sequence of compact operators. Because
$ \mathcal{L}_c (E) $ is a closed subset of $ \mathcal{L}(E) $, we have
proved (i). Property (ii) follows from the equality
\[
[A(t),P^+ (A(s))] = [A(t),P^+ (A(t))] + [A(t),P^+ (A(s)) - P^+ (A(t))],
\]
where the first summand 
of the right member is zero and the second one is
compact by hypothesis. In particular, the equality holds for $ s > a $, so
we finish our proof by taking the limit on the left member as
$ s\rightarrow +\infty $.
\end{proof}
\subsection{The spectral flow and the Fredholm index of $ F_A $}
Given a continuous, bounded path $ A_t \in\mathcal{L}(E) $, we
denote by $ F_A $ the differential operator
\[
F_A \colon W\sp{1,p} (\R,E)\rightarrow L\sp p (\R,E), \ \ 
F_A (u) = \frac{du}{dt} - A(\cdot)u.
\]
\begin{theorem}
\label{thm:index-FA}
Let $ A $ be an asymptotically hyperbolic and essentially splitting path.
Then $ F_A $ is a Fredholm operator and 
\[
\ind(F_A) = [P^- (A(-\infty)) - P^- (A(+\infty))]
\]
\end{theorem}
A.~Abbondandolo and P.~Majer proved it in 
\cite[Theorem~6.3]{AM03} where $ E $ is a Hilbert space and $ p = 2 $. However,
the theorem, like much of the content of their work, can be generalised to 
Banach spaces as in \cite[Theorem~3.3]{Gar08}.
\begin{theorem}
\label{thm:sf-ess-split}
Let $ A $ be an asymptotically hyperbolic, essentially splitting and 
essentially hyperbolic path. Then,
\[
\Sf(A) = - [P^- (A(-\infty)) - P^- (A(+\infty))]
\]
\end{theorem}
\begin{proof}
Let $ \delta > 0 $ as in (\ref{sf_hyp}). By Lemma \ref{essential_splitting},
the constant path $ P\equiv P\sp + (A(\delta)) $ is an s-section for
$ P\sp + (A_t) $ on $ [-\delta,\delta] $ in the sense of Definition
\ref{defn:section}. Hence,
\[
\Sf(A) = [P\sp + (A(\delta)) - P\sp + (A(-\delta))].
\]
Because $ A $ is hyperbolic on $ (-\infty-\delta]\cup [\delta,+\infty) $, the 
path $ P^+ (A_t) $ is continuous on this subset. 
By (iii) of Theorem \ref{thm:relative-dimension},
\[
[P^- (A(-\infty)) - P^- (A(+\infty))] = 
- [P^+ (A(\delta)) - P^+ (A(-\delta))] = -\Sf(A).
\]
\end{proof}
From Theorem \ref{thm:sf-ess-split} and Theorem \ref{thm:index-FA} we
have the final result
\begin{theorem}
\label{thm:final}
If $ A $ is essentially hyperbolic, essentially splitting and
asymptotically hyperbolic, then 
\begin{equation}
\label{eq:final}
\ind(F_A) = -\Sf(A).
\end{equation}
\end{theorem}
Let $ A(t) = A_0 (t) + K(t) $, where $ A_0 (t) $ is hyperbolic and
$ A_0,A $ are asymptotically hyperbolic. $ K(t) $ is compact and 
\begin{gather*}
A_0 (t) E\sp - \subseteq E\sp -,\ \ A_0 (t) E\sp + \subseteq E\sp +;\\
E\sp - = E\sp - (A_0 (\pm\infty)),\ \ E\sp + = E\sp + (A_0 (\pm\infty)).
\end{gather*}
The second line tells us that 
$ P\sp + (A_0 (+\infty)) = P\sp + (A_0 (-\infty)) $ and thus 
$ P\sp + (A(+\infty)) - P\sp + (A(-\infty)) $ is compact. From the
first line, it follows that $ [A(t),P\sp + (A(+\infty))] $ is compact.
Thus, by Theorem \ref{thm:final}, we confirm the guess of A.~Abbondandolo and 
P.~Majer in \cite[\S7]{AM03}, that for paths satisfying the hypotheses 
of \cite[Theorem~E]{AM03}, corresponding to those listed above, the relation
(\ref{eq:final}) holds. 

When $ A $ is not essentially splitting, the authors provided in 
\cite[Example~6,7]{AM03} counterexamples to (\ref{eq:final}).
\vfill
\nocite{*}
\bibliographystyle{plain}

\end{document}